\documentclass[11pt,a4paper]{article}
\usepackage[T1]{fontenc}
\usepackage[utf8]{inputenc}
\usepackage{lmodern,amsmath,amssymb,amsthm,mathtools,microtype}
\usepackage[margin=27mm]{geometry}
\usepackage{booktabs,longtable,array,xcolor,needspace}
\usepackage[colorlinks=true,linkcolor=blue!40!black,citecolor=blue!40!black,urlcolor=blue!40!black]{hyperref}
\hypersetup{pdftitle={Local obstructions and certified computations for perfect powers in power sums},pdfauthor={Akos Pinter},pdfsubject={Perfect powers in power sums}}
\newtheorem{theorem}{Theorem}[section]
\newtheorem{lemma}[theorem]{Lemma}
\newtheorem{corollary}[theorem]{Corollary}
\newtheorem{proposition}[theorem]{Proposition}
\theoremstyle{remark}
\newtheorem{remark}[theorem]{Remark}
\newtheorem{example}[theorem]{Example}
\newcommand{\Z}{\mathbb Z}
\newcommand{\N}{\mathbb N}
\newcommand{\Q}{\mathbb Q}
\DeclareMathOperator{\lcm}{lcm}

\title{\bfseries Local obstructions and certified computations for perfect powers in power sums}
\author{\'Akos Pint\'er\\[3pt]\small University of Debrecen, Hungary}
\date{}
\begin{document}
\maketitle
\begin{abstract}
Let $S_k(x)=1^k+\cdots+x^k$. We study perfect-power values of
$S_k(x)$ by combining local valuations with finite congruence
certificates. For an irregular Bernoulli pair $(p,r)$ with
$e=v_p(B_r)\ge1$, we prove a valuation formula uniform for
$k\equiv r+1\pmod{p^e(p-1)}$, subject to an explicit condition on
$v_p(x(x+1))$. This gives explicit families of nonexistence results
in which both $x$ and $k$ are unbounded. We also determine all
solutions of $S_k(x)=y^n$ for $3\le x\le10^6$ with
$x\equiv0,3\pmod4$, leaving $k$ and $n\ge2$ unrestricted and
including squares. This extends Hajdu's fixed-$x$ classification
in the same residue classes. A uniform restriction on the normalized
odd part of $S_k(x)$ and finite certificates exclude every $k\ge6$;
the remaining exponents are checked by exact integer arithmetic.
The certificates admit independent verification. We further prove a
limitation of congruence tests of bounded prime-power depth and solve
the equation for every $x\equiv28\pmod{5\,110\,664}$.
\end{abstract}

\section{Introduction and main results}
For positive integers $x$ and $k$, put
\[
 S_k(x)=\sum_{j=1}^{x}j^k.
\]
We consider the Diophantine equation
\begin{equation}\label{eq:main}
 S_k(x)=y^n,
 \qquad x,k,y\in\N,\quad n\in\Z_{\ge2}.
\end{equation}
The convention here includes the last term $x^k$.
For a sum ending at $(X-1)^k$, all statements below translate by
$x=X-1$.

Sch\"affer \cite{Sch} proved that, for each fixed pair $(k,n)$ outside
\begin{equation}\label{eq:exceptional-pairs}
 \{(1,2),(3,2),(3,4),(5,2)\},
\end{equation}
equation~\eqref{eq:main} has only finitely many solutions in $x,y$.
His conjecture asserts that, outside these four pairs, the only
solution with $x>1$ is
\[
 (x,k,n,y)=(24,2,2,70).
\]
The identity
\begin{equation}\label{eq:cubes}
 S_3(x)=\left(\frac{x(x+1)}2\right)^2
\end{equation}
accounts for an infinite family that occurs throughout the present paper.

The paper develops two complementary forms of local exclusion.
The first is a valuation theorem attached to irregular Bernoulli
pairs, uniform over entire arithmetic progressions of $k$ and valid
for unbounded $x$. The second is a complete fixed-$x$ classification
through $10^6$, certified by finite residue calculations with no
upper bound on $k$. We also describe a structural limitation of
congruence tests with bounded prime-power depth. The known valuation
formulas used in the computation are stated explicitly, and the
$2$-adic formula is supplied with an elementary proof.

Effective finiteness results were obtained by Gy\H{o}ry, Tijdeman and
Voorhoeve \cite{GTV80}: for fixed $k\ge2$ with $k\notin\{3,5\}$,
the nontrivial solutions can be effectively determined even when $n$
is unknown. Voorhoeve, Gy\H{o}ry and Tijdeman \cite{VGT79} also
studied polynomial perturbations of power sums. Quantitative
refinements include Pint\'er's bounds for the power exponent in
terms of $k$ \cite{P97} and the bounds for the number of solutions
obtained by Brindza and Pint\'er \cite{BP00}. Gy\H{o}ry and Pint\'er
\cite{GP03} give a detailed account of these finiteness results,
quantitative bounds and early complete classifications.

Several works resolve the equation for specified summand exponents.
Jacobson, Pint\'er and Walsh \cite{JPW03} determined the square values
for even $2\le k\le58$. Bennett, Gy\H{o}ry and Pint\'er \cite{BGP04}
proved Sch\"affer's conjecture for every $1\le k\le11$, with arbitrary
$n\ge2$. Pint\'er \cite{P07} proved that there are no solutions with
$x>1$ when $k$ is odd, $1\le k<170$, and $n>4$ is even.
Recent work of Chen, Choi and Houle \cite{CCH26} develops a descent
to coupled generalized Fermat equations and a modular approach for
even $k\le42$.

A related direction allows the initial term of the sum to vary.
Bennett, Patel and Siksek \cite{BPS16} solved the equations involving
three consecutive fifth or sixth powers, and in \cite{BPS17} they
classified perfect powers represented by sums of $2$ to $50$
consecutive cubes. Patel \cite{Patel18} treated sums of $2$ to $10$
consecutive squares. Extensions to arithmetic progressions include
the work of Kundu and Patel \cite{KP21} on squares, Arg\'aez-Garc\'ia
and Patel \cite{AP20} on seven cubes, and Coppola, Curc\'o-Iranzo,
Khawaja, Patel and \"Ulkem \cite{CCKPU24} on nine cubes. These
arithmetic-progression results impose bounds on the common difference
and suitable coprimality conditions. For a broader account of power
values of power sums and the methods used to study them, see the
survey \cite{CCKPU24survey}.

Density results give a further perspective. Patel and Siksek
\cite{PS} proved that, for each fixed even summand exponent and
fixed nonzero common difference, the set of lengths admitting a
perfect-power value has natural density zero. Gy\H{o}ry, Patel,
Pint\'er and Siksek \cite{GPPS26} obtained a related result for
character-twisted sums under an irreducibility hypothesis on a
generalized Bernoulli polynomial. We compare these results with
our local density statements in Section~\ref{sec:odd}.

A complementary approach is to impose conditions on $x$ while
leaving the exponents $k$ and $n$ unrestricted. Hajdu \cite{H15}
pursued this approach to~\eqref{eq:main}. His
Theorem~2.1 establishes Sch\"affer's conjecture under explicit
congruence conditions on $x$, and his Theorem~2.3 determines all
solutions for
\[
 x<25,\qquad x\equiv0,3\pmod4.
\]
His method combines exact $2$-adic and $3$-adic valuations of power
sums with polynomial--exponential congruences. The corrigendum
\cite{H16} corrects the proof of the $3$-adic valuation formula;
the formula and the resulting theorems remain valid.

B\'erczes, Hajdu, Miyazaki and Pink \cite[Theorem~2.1]{BHMP}
subsequently treated every $x<25$ under the assumption $n\ge3$.
They proved that the complete solution set in that range is
\[
 (x,k,n,y)=(1,k,n,1)\quad\text{or}\quad(8,3,4,6).
\]
To treat the remaining classes $x\equiv1,2\pmod4$, they combine
lower bounds for linear forms in two logarithms with congruence
arguments and computation.

\Needspace{9\baselineskip}
We first state the uniform valuation result. Let $B_j$ denote the
Bernoulli numbers with $B_1=-1/2$, and let $v_p$ be the usual
$p$-adic valuation on nonzero rational numbers. The role of Bernoulli
numerators in valuations of power sums is already present in
Lengyel's work \cite[Theorem~5]{L}. The following theorem combines
that principle with Kummer's congruences and gives an explicit
condition that is uniform throughout the stated progression of $k$.

\begin{theorem}[Uniform valuation on an exponent progression]
\label{thm:kummer-uniform}
Let $p\ge7$ be a prime and let $r$ be an even integer satisfying
\[
 4\le r\le p-3,\qquad e=v_p(B_r)\ge1.
\]
Suppose that $x,k$ are positive integers with
\[
 k\equiv r+1\pmod{p^e(p-1)},\qquad
 a=v_p(x(x+1))\ge1,\qquad 2a>e.
\]
Then
\begin{equation}\label{eq:kummer-uniform-valuation}
 v_p(S_k(x))=2a+e.
\end{equation}
If $e$ is odd, $S_k(x)$ is not a square. If $e$ is odd and, in addition,
$x\equiv3,4\pmod8$, equation~\eqref{eq:main} has no solution.
\end{theorem}

The congruence on $k$ makes $k$ odd and $k\ge r+1\ge5$.
For $e=1$, the only condition on the $p$-adic depth of the endpoint
is $p\mid x(x+1)$. In particular, the theorem applies to both
classes $x\equiv1,2\pmod4$ as well as to the two classes used in
the finite computation, although its final perfect-power exclusion
uses the additional condition modulo $8$.

\begin{corollary}\label{cor:kummer-families}
Equation~\eqref{eq:main} has no solution under either of the
following sets of congruence conditions:
\begin{align*}
 k&\equiv33\pmod{1332},&
 x&\equiv36,147,148,259\pmod{296};\\
 k&\equiv13\pmod{476\,790},&
 x&\equiv691,2763,2764,4836\pmod{5528}.
\end{align*}
\end{corollary}

Both $x$ and $k$ are unbounded in these families. The first family
contains the subprogression $k\equiv33\pmod{2664}$, which satisfies
$k\equiv1\pmod8$ and therefore survives the initial necessary
condition for squares used below. We prove the theorem and its
corollary in Section~\ref{subsec:kummer-obstructions}.

Our computational result extends Hajdu's classification in the two
classes $x\equiv0,3\pmod4$ to $3\le x\le10^6$, retaining all
power exponents $n\ge2$, including squares. Although the range of
$x$ is finite, both $k$ and $n$ remain unrestricted. The proof
reduces the complete classification to finite congruence
certificates and exact integer arithmetic.
\Needspace{13\baselineskip}
\begin{theorem}\label{thm:finite}
Let $3\le x\le10^6$ and $x\equiv0,3\pmod4$.
All solutions of~\eqref{eq:main} are
\begin{equation}\label{eq:universal}
 (k,n,y)=\left(3,2,\frac{x(x+1)}2\right)
\end{equation}
for every such $x$, together with the following nine solutions:
\begin{equation}\label{eq:exceptions}
 \begin{gathered}
 (x,k,n,y)=(8,1,2,6),\ (8,3,4,6),\ (24,2,2,70),\\
 (288,1,2,204),\ (288,3,4,204),\\
 (9800,1,2,6930),\ (9800,3,4,6930),\\
 (332928,1,2,235416),\ (332928,3,4,235416).
 \end{gathered}
\end{equation}
\end{theorem}

\Needspace{5\baselineskip}
The theorem gives $500\,009$ solution quadruples for the $500\,000$ values
of $x$ in its range. It extends Hajdu's fixed-$x$ theorem in the
same two residue classes. Extending the full range of \cite{BHMP}
would also require treating $x\equiv1,2\pmod4$, where the
$2$-adic valuation does not provide the exponent reduction used here.

A useful refinement in the square case is the uniform congruence
for the normalized odd part of $S_k(x)$ in
Lemma~\ref{lem:uniform-two-square}. For odd $k\ge5$, it implies
that a square value requires $k\equiv1\pmod8$, uniformly for
every $x\equiv0,3\pmod4$. Hajdu already uses this restriction
for $x\equiv3,4\pmod8$ in the proof of \cite[Theorem~2.1]{H15}.
The normalized formulation gives a common starting point for all
the square branches in the present computation and avoids the use
of large powers of $2$ as certificate moduli.

For each $x$, the known valuation formula reduces the perfect-power
problem to finitely many equations $S_k(x)=z^r$ with $r$ prime.
For each resulting pair $(x,r)$, periodicity in $k$ then turns a finite list of
congruence conditions into a certificate excluding every $k\ge6$.
Proposition~\ref{prop:certificate} makes this reduction explicit.
All the certificate moduli are at most $991$; the remaining cases
$1\le k\le5$ are checked by exact integer arithmetic, including
every possible composite power exponent. The electronic supplement
contains the full lists of moduli and a verifier that recomputes
the residue intersections directly. Verification therefore does
not depend on the search that produced the certificates.

Prime-power moduli have a specific role in this argument. Since
$S_1(9800)=6930^2$, a square-residue sieve using any finite collection
of prime moduli retains an infinite arithmetic progression of
exponents congruent to $1$ modulo the corresponding unit-group
periods. Section~\ref{sec:finite} shows
how the modulus $625$ distinguishes these large exponents from
the solution with $k=1$. This example illustrates Proposition~\ref{prop:bounded-depth},
which gives a general obstruction to eliminating all large exponents
by congruence tests of bounded prime-power depth. The proposition
also explains why the full arithmetic depth of a normalized test
must be taken into account.

The same local approach also gives a complete classification on
an infinite arithmetic progression.
\begin{theorem}\label{thm:progression}
Suppose that
\[
 x\equiv28\pmod{5\,110\,664},
 \qquad 5\,110\,664=8\cdot13\cdot157\cdot313.
\]
Then all solutions of~\eqref{eq:main} are given by~\eqref{eq:universal}.
\end{theorem}

Theorem~\ref{thm:progression} shows how a certificate for one value
of $x$ can hold uniformly over an arithmetic progression. Hajdu
already noted that his method permits further congruence
conditions on $x$ \cite[Remark~1]{H15}. Here we give a fully
explicit instance, with all the residue tables included in the proof.

Section~\ref{sec:sieve} develops the valuation reduction, the
limitation on congruence tests of bounded depth, and the certificate
criterion. Sections~\ref{sec:finite} and~\ref{sec:progression}
prove Theorems~\ref{thm:finite} and~\ref{thm:progression},
respectively. The worked example in Section~\ref{subsec:example224}
gives the complete argument for $x=224$.
Section~\ref{sec:odd} develops the local results independent of the
finite computation. It includes an elementary pairing proof of a
consequence of Lengyel's valuation theorem, local density statements,
and a fixed-exponent Bernoulli obstruction. The final subsection
proves Theorem~\ref{thm:kummer-uniform} and
Corollary~\ref{cor:kummer-families}, using the full Kummer
congruences to obtain uniformity in $k$.

\section{Valuation reduction and periodic congruences}\label{sec:sieve}
For a prime $\ell$ and a nonzero integer $N$, write $v_\ell(N)$ for
the exponent of $\ell$ in $N$. We use the same notation for the
usual extension to nonzero rational numbers.

\subsection{The possible power exponents}
The following formula is due to MacMillan and Sondow \cite{MS}; it
is reproduced as Lemma~3.1 in \cite{H15}. We include an elementary
proof to make the valuation reduction self-contained.
\begin{lemma}\label{lem:two}
For positive integers $x,k$, let $a=v_2(x(x+1))$. Then
\begin{equation}\label{eq:two}
 v_2(S_k(x))=
 \begin{cases}
 a-1,& k=1\text{ or }k\text{ is even},\\
 2a-2,& k\ge3\text{ is odd}.
 \end{cases}
\end{equation}
\end{lemma}
\begin{proof}
For $k=1$, use $S_1(x)=x(x+1)/2$. Suppose henceforth that $k\ge2$,
let $m$ be the even member of $\{x,x+1\}$, and put
$A_j(m)=\sum_{i=1}^{m-1}i^j$. Thus $v_2(m)=a$.
If $a=1$, the sum $A_k(m)$ contains $m/2$ odd terms, an odd
number, so $A_k(m)$ is odd.

For even $k\ge2$, proceed by induction on $a$. The case $a=1$
has just been proved. For $a\ge2$, pairing $i$ with $m-i$ gives
\[
 A_k(m)=\left(\frac m2\right)^k+
 \sum_{i=1}^{m/2-1}\bigl(i^k+(m-i)^k\bigr)
 \equiv2A_k(m/2)\pmod{2^a}.
\]
The middle term is divisible by $2^a$, since $k(a-1)\ge a$.
Induction gives $v_2(A_k(m/2))=a-2$, hence
$v_2(A_k(m))=a-1$.

For odd $k\ge3$ and $a\ge2$, the same pairing yields
\[
 A_k(m)=\left(\frac m2\right)^k+
 kmA_{k-1}(m/2)+m^2T
\]
for some integer $T$. The even case already proved shows that
the second term has valuation $a+(a-2)=2a-2$. The first term has
valuation $k(a-1)>2a-2$, and $m^2T$ is divisible by $2^{2a}$.
Thus $v_2(A_k(m))=2a-2$.

Finally, $S_k(x)=A_k(m)$ for $x=m-1$, and
$S_k(x)=A_k(m)+m^k$ for $x=m$. In the latter case $m^k$ has
strictly larger valuation than $A_k(m)$, so the valuation is
unchanged. This also covers $a=1$, and completes the proof.
\end{proof}

\begin{corollary}\label{cor:exponents}
Suppose $x\equiv0,3\pmod4$, and set
$b=v_2(x(x+1))-1\ge1$. If~\eqref{eq:main} holds, then $n\mid b$
when $k=1$ or $k$ is even, and $n\mid2b$ when $k\ge3$ is odd.
In particular, $S_k(x)$ is an $r$-th power for some prime $r\mid2b$.
If $r\nmid b$, such a prime-power representation requires odd $k\ge3$.
\end{corollary}
\begin{proof}
The valuation of $y^n$ is divisible by $n$. If $r\mid n$ is prime,
then $y^n=(y^{n/r})^r$, and~\eqref{eq:two} gives the assertions.
\end{proof}

\subsection{Normalized odd parts modulo eight}
The valuation alone records the power of $2$ dividing a power sum.
For squares, its odd part supplies a further uniform restriction,
which avoids using large powers of $2$ as certificate moduli.

\begin{lemma}[Uniform $2$-adic square obstruction]\label{lem:uniform-two-square}
Let $x\equiv0,3\pmod4$, let $m$ be the even member of $\{x,x+1\}$,
and write $m=2^a d$, where $a\ge2$ and $d$ is odd. For an even integer
$r\ge4$, put
\[
 c_r=\begin{cases}1,&r\equiv0\pmod4,\\5,&r\equiv2\pmod4.\end{cases}
\]
Then
\begin{align}
 \frac{S_k(x)}{2^{a-1}}&\equiv d c_k\pmod8
       &&\text{if $k\ge4$ is even},\label{eq:even-normalized-two}\\
 \frac{S_k(x)}{2^{2a-2}}&\equiv k c_{k-1}\pmod8
       &&\text{if $k\ge5$ is odd}.\label{eq:odd-normalized-two}
\end{align}
In particular, for odd $k\ge5$, the normalized odd part of $S_k(x)$ is
congruent to $1,7,5,3$ modulo $8$ according as $k$ is congruent to
$1,3,5,7$ modulo $8$.
Consequently, if $S_k(x)$ is a square and $k\ge5$, then $k$ belongs to
one of the following classes modulo $8$:
\[
 \begin{array}{c|c}
 \text{condition on $a,d$}&\text{admissible classes of $k$ modulo $8$}\\\hline
 a\text{ odd},\ d\equiv1\pmod8&\{0,1,4\}\\
 a\text{ odd},\ d\equiv5\pmod8&\{1,2,6\}\\
 \text{all other cases}&\{1\}.
 \end{array}
\]
These are necessary conditions; admissibility does not assert that the
power sum is a square.
\end{lemma}

\begin{proof}
Write $A_r(u)=\sum_{j=1}^{u-1}j^r$ for $r\ge1$ and
$A_0(u)=u-1$. We use Lemma~\ref{lem:two} in the form
\[
 v_2(A_r(u))=\begin{cases}
 v_2(u)-1,&r=1\text{ or }r\text{ even},\\
 2v_2(u)-2,&r\ge3\text{ odd},
 \end{cases}
 \qquad(2\mid u).
\]
Indeed, this follows by applying that formula to the power sum with
upper limit $u-1$.

We first prove that, for every even $r\ge4$ and $a\ge2$,
\begin{equation}\label{eq:pure-two-even-unit}
 \frac{A_r(2^a)}{2^{a-1}}\equiv c_r\pmod8.
\end{equation}
At $a=2$, one has $A_r(4)=1+2^r+3^r$; reducing modulo $16$ gives
$2$ if $r\equiv0\pmod4$ and $10$ if $r\equiv2\pmod4$, as required.
Put $N=2^a$. Expansion of the block from $N$ to $2N-1$ yields
\[
 A_r(2N)=2A_r(N)+
 \sum_{h=1}^{r-1}\binom rh N^h A_{r-h}(N)+N^{r+1}.
\]
The term with $h=1$ has valuation at least
\[
 1+a+(2a-2)=3a-1,
\]
because $r$ is even and $r-1\ge3$ is odd. For $2\le h\le r-1$,
the valuation is at least $ha+(a-1)\ge3a-1$.
The final term has valuation $(r+1)a\ge3a-1$.
Since $3a-1\ge a+3$ for $a\ge2$, it follows that
\[
 A_r(2^{a+1})\equiv2A_r(2^a)\pmod{2^{a+3}}.
\]
Dividing by $2^a$ proves the induction step in
\eqref{eq:pure-two-even-unit}.

The same expansion into $d$ blocks gives
\[
\begin{split}
 A_r(dN)={}&dA_r(N)+
 \sum_{h=1}^{r-1}\binom rh N^h
 \left(\sum_{t=0}^{d-1}t^h\right)A_{r-h}(N)\\
 &+N^{r+1}\sum_{t=0}^{d-1}t^r.
\end{split}
\]
The preceding estimates show that every term following $dA_r(N)$
is divisible by $2^{3a-1}$, hence by $2^{a+2}$.
Using \eqref{eq:pure-two-even-unit}, we obtain
\begin{equation}\label{eq:even-A-unit}
 A_r(m)\equiv2^{a-1}d c_r\pmod{2^{a+2}}.
\end{equation}
If $x=m-1$, then $S_r(x)=A_r(m)$; if $x=m$, then
$S_r(x)=A_r(m)+m^r$. Since $ra\ge a+2$, the added term does not
change \eqref{eq:even-A-unit}. This proves
\eqref{eq:even-normalized-two}.

Now let $k\ge5$ be odd. Replacing $j$ by $m-j$ gives
\[
 2A_k(m)=kmA_{k-1}(m)+
 \sum_{h=2}^{k}(-1)^{k-h}\binom kh m^h A_{k-h}(m).
\]
The $h=2$ term has valuation at least $2a+(2a-2)=4a-2$, since
$k-2\ge3$ is odd. For $3\le h\le k-1$, the valuation is at least
$ha+(a-1)\ge4a-1$. The final term has valuation $ka$ because
$A_0(m)=m-1$ is odd. All these lower bounds are at least $2a+2$:
this uses $a\ge2$ and $k\ge5$.
After division by $2$, therefore,
\[
 A_k(m)\equiv\frac{km}{2}A_{k-1}(m)\pmod{2^{2a+1}}.
\]
Apply \eqref{eq:even-A-unit} with $r=k-1$. Since $m=2^a d$, this gives
\[
 A_k(m)\equiv2^{2a-2}k d^2 c_{k-1}\pmod{2^{2a+1}}.
\]
Again the possible added term $m^k$ has valuation at least $2a+1$.
As $d^2\equiv1\pmod8$, this proves \eqref{eq:odd-normalized-two}.

An odd square is congruent to $1$ modulo $8$, and a nonzero square
has even $2$-adic valuation. The odd-$k$ table thus permits only
$k\equiv1\pmod8$. For even $k$, the valuation $a-1$ must be even,
so $a$ is odd; the condition on the odd part is $d c_k\equiv1\pmod8$.
It gives $d\equiv1\pmod8$ when $k\equiv0\pmod4$, and
$d\equiv5\pmod8$ when $k\equiv2\pmod4$. These are exactly the
classes listed in the statement.
\end{proof}

\subsection{Periodicity and power residues}
For a prime power $q=\ell^a$, put
\begin{equation}\label{eq:lambda}
 \lambda(q)=
 \begin{cases}
 2^{a-2},&\ell=2\text{ and }a\ge3,\\
 (\ell-1)\ell^{a-1},&\text{otherwise}.
 \end{cases}
\end{equation}
This is the exponent of the group of units modulo $q$.
For a general integer $M\ge2$, we use the Carmichael function
\[
 \lambda(M)=\lcm_{p^a\Vert M}\lambda(p^a).
\]

\begin{lemma}\label{lem:periodic}
Fix $x\ge1$ and $q=\ell^a$. If $k,k'\ge a$ and
$k\equiv k'\pmod{\lambda(q)}$, then
\[
 S_k(x)\equiv S_{k'}(x)\pmod q.
\]
\end{lemma}
\begin{proof}
If $\ell\mid j$, both $j^k$ and $j^{k'}$ vanish modulo $q$.
If $\ell\nmid j$, their equality follows from the definition of
$\lambda(q)$. Sum over $1\le j\le x$.
\end{proof}

\subsection{A limitation of tests of bounded prime-power depth}
The periodicity just established also describes an obstruction to
excluding large exponents by a finite collection of shallow
congruence tests. The following elementary proposition concerns
the congruences themselves, without additional valuation restrictions.

\begin{proposition}\label{prop:bounded-depth}
Let $x,d,w,s$ be positive integers with $s\ge2$ and
$S_d(x)=w^s$. Let $m_1,\ldots,m_t\ge2$, where $t\ge1$, and put
$M=\lcm(m_1,\ldots,m_t)$. Suppose that
\[
 v_p(M)\le d\qquad\text{for every prime $p\le x$}.
\]
Then every integer $k\ge d$ with
$k\equiv d\pmod{\lambda(M)}$ satisfies
\[
 S_k(x)\equiv w^s\pmod{m_i}\qquad(1\le i\le t).
\]
In particular, these congruences alone leave infinitely many
possible exponents $k$.
\end{proposition}
\begin{proof}
Fix $p^a\Vert M$ and $1\le j\le x$. If $p\nmid j$, the
divisibility $\lambda(p^a)\mid\lambda(M)$ gives
$j^k\equiv j^d\pmod{p^a}$. If $p\mid j$, then $p\le x$ and
$a\le d\le k$, so both powers vanish modulo $p^a$.
The Chinese remainder theorem and summation therefore give
\[
 S_k(x)\equiv S_d(x)=w^s\pmod M.
\]
The stated progression of $k$ contains arbitrarily large integers.
\end{proof}

\begin{corollary}\label{cor:depth-squares}
For any fixed $x\ge1$, a finite family of direct tests
$S_k(x)\equiv Y^2\pmod m$ cannot exclude all sufficiently large
$k$ unless at least one of its moduli is divisible by $p^4$ for
some prime $p\le x$.
\end{corollary}
\begin{proof}
Apply Proposition~\ref{prop:bounded-depth} with $d=3$ and
$w=x(x+1)/2$, using~\eqref{eq:cubes}.
\end{proof}

\begin{remark}
This restriction concerns the full arithmetic depth of the tests.
For example, if $v_2(S_k(x))=2h$, the normalized square condition
$S_k(x)/2^{2h}\equiv1\pmod8$ is a congruence modulo $2^{2h+3}$
for $S_k(x)$ itself. Thus the normalized condition of
Lemma~\ref{lem:uniform-two-square} already uses greater $2$-adic
depth than the displayed modulus $8$ suggests.
\end{remark}

\subsection{Power residues}
For a prime $r$, let
\[
 \mathcal R_r(q)=\{z^r\bmod q:z\in\Z\}.
\]
These sets can be determined by direct enumeration. The verifier
uses the following equivalent test, which avoids storing a set of
$q$ residues.
\begin{lemma}\label{lem:residues}
Let $q=\ell^a$, let $r$ be prime, and choose $0\le c<q$.
The residue $0$ belongs to $\mathcal R_r(q)$. For $c\ne0$, write
$c=\ell^v u$ with $\ell\nmid u$, and put $s=a-v$.
Then $c\in\mathcal R_r(q)$ if and only if $r\mid v$ and the
following applicable condition holds:
\begin{enumerate}
\item if $\ell$ is odd, then
\[
 u^{\varphi(\ell^s)/\gcd(r,\varphi(\ell^s))}
 \equiv1\pmod{\ell^s};
\]
\item if $\ell=2$ and $r$ is odd, there is no further condition;
\item if $\ell=r=2$, then $u\equiv1\pmod{2^{\min(3,s)}}$.
\end{enumerate}
\end{lemma}
\begin{proof}
If $z^r\equiv c\pmod{\ell^a}$ and $c\ne0$, then
$r v_\ell(z)=v<a$. Dividing out $\ell^v$ reduces the problem to
whether $u$ is an $r$-th power among the units modulo $\ell^s$.
Conversely, a unit solution gives a solution for $c$ after multiplication
by $\ell^{v/r}$.
For odd $\ell$, the unit group is cyclic, giving the first criterion.
For $\ell=2$, taking an odd power is an automorphism of the unit
group. Finally, an odd unit is a square modulo $2^s$ precisely when
it is $1$ modulo $2^{\min(3,s)}$.
\end{proof}

\subsection{A finite certificate criterion}
Let $K\ge5$ and $x\equiv0,3\pmod4$ be fixed, set
$b=v_2(x(x+1))-1$, and let $r\mid2b$ be prime. Choose prime powers
$q_1,\ldots,q_t$, all with prime-power exponent at most $K$.
If $r=2$, set $M_0=8$ and let $A_0$ be the set of classes given by
Lemma~\ref{lem:uniform-two-square}. If $r$ is odd, define
\[
 M_0=2,\qquad
 A_0=\begin{cases}\{0,1\},&r\mid b,\\\{1\},&r\nmid b.
 \end{cases}
\]
(For the odd primes under consideration, $r\mid2b$ in fact implies
$r\mid b$.)
Suppose $A_{i-1}\subseteq\{0,\ldots,M_{i-1}-1\}$ has been defined.
Put
\[
 m_i=\lambda(q_i),\qquad
 M_i=\lcm(M_{i-1},m_i).
\]
For an integer $c$, let $\kappa_i(c)$ be the unique integer in
$[K,K+m_i-1]$ congruent to $c$ modulo $m_i$. Define
\begin{equation}\label{eq:recurrence}
 \begin{split}
 A_i=\{\,c={}&a+jM_{i-1}:a\in A_{i-1},\quad
 0\le j<M_i/M_{i-1},\\[-1mm]
 &S_{\kappa_i(c)}(x)\bmod q_i\in\mathcal R_r(q_i)\,\}.
 \end{split}
\end{equation}
Every element in this set lies between $0$ and $M_i-1$.

\begin{proposition}\label{prop:certificate}
With the notation above, every solution of $S_k(x)=z^r$ with
$k\ge K$ satisfies $k\bmod M_i\in A_i$ for each
$0\le i\le t$. In particular, if $A_t$ is empty, no such solution
exists.
\end{proposition}
\begin{proof}
Lemma~\ref{lem:uniform-two-square} gives the initial classes when
$r=2$, and Corollary~\ref{cor:exponents} gives them when $r$ is odd.
At the $i$-th step, every residue modulo $M_i$ lying above a class
in $A_{i-1}$ has exactly the form enumerated
in~\eqref{eq:recurrence}. Lemma~\ref{lem:periodic} identifies the
power sum for $k$ with that for $\kappa_i(c)$ modulo $q_i$.
The latter must be an $r$-th power residue. Induction proves the
claim.
\end{proof}

The order of the moduli affects the running time but not the set
of simultaneous conditions. The search chooses the moduli adaptively.
The independent verifier reads each list in its recorded order and
performs every lift in~\eqref{eq:recurrence}, without period compression.

\section{Proof of Theorem~\ref{thm:finite}}\label{sec:finite}
For $3\le x\le10^6$ with $x\equiv0,3\pmod4$, the integer
$b=v_2(x(x+1))-1$ lies between $1$ and $18$. Its distribution and
the prime exponents required by Corollary~\ref{cor:exponents} are
shown in Table~\ref{tab:branches}. Thus there are exactly $591\,804$
pairs $(x,r)$ to consider, with $r\in\{2,3,5,7,11,13,17\}$.
The numbers of branches for these primes, in order, are
$500\,000$, $71\,428$, $16\,130$, $3936$, $244$, $62$ and $4$.
Every original power exponent satisfies $n\le36$ in this range.

\begin{table}[htbp]
\centering
\caption{The reduction to prime exponents in the range of
Theorem~\ref{thm:finite}.}\label{tab:branches}
\begin{tabular}{crcr}
\toprule
$b$&Number of values of $x$&Primes dividing $2b$&Branches\\
\midrule
1&250000&2&250000\\
2&125000&2&125000\\
3&62500&2, 3&125000\\
4&31250&2&31250\\
5&15626&2, 5&31252\\
6&7812&2, 3&15624\\
7&3906&2, 7&7812\\
8&1954&2&1954\\
9&976&2, 3&1952\\
10&488&2, 5&976\\
11&244&2, 11&488\\
12&122&2, 3&244\\
13&62&2, 13&124\\
14&30&2, 7&60\\
15&16&2, 3, 5&48\\
16&8&2&8\\
17&4&2, 17&8\\
18&2&2, 3&4\\
\midrule
Total&500000&&591804\\
\bottomrule
\end{tabular}
\end{table}

\subsection{Exclusion of all exponents \texorpdfstring{$k\ge6$}{k >= 6}}
For each pair, the electronic file \path{certificates.tsv} contains
an ordered finite list of prime-power moduli. There are $1\,221\,859$
modulus entries in total, using $87$ distinct moduli. The largest
modulus is $991$. Every modulus has prime-power exponent at most
$6$, with exponent $6$ occurring at $729=3^6$.
We apply Proposition~\ref{prop:certificate} with $K=6$.
The verifier constructs all sets in~\eqref{eq:recurrence}, starting
from the uniform square restriction or the parity condition as
appropriate, and using Lemmas~\ref{lem:periodic} and~\ref{lem:residues}.
The final set is empty for every one of the $591\,804$ pairs.
Consequently there is no solution of~\eqref{eq:main} with $k\ge6$
in the stated range of $x$.

The largest value of $\lcm(M_0,\lambda(q_1),\ldots,\lambda(q_t))$
over the certificate lists is $19\,656\,000$, attained at $(x,r)=(266631,2)$.
Here $M_0=8$ for squares and $M_0=2$ for odd prime exponents.
The surviving classes are stored explicitly, and every lift to the
next period is tested. In total, the verifier examines $74\,389\,941$
lifts and performs $39\,009\,676$ modular power-residue tests.
The largest surviving set contains $9315$ classes, occurring at
$(x,r)=(167399,2)$ after the modulus $729$.
Thus the verification imposes no upper bound on $k$.
Representative lists are given in
Appendix~\ref{app:certificates}.

As a small illustration, take $x=128$ and $r=3$. The modulus $7$
alone suffices. Since $128=18\cdot7+2$, complete blocks modulo $7$
give
\[
\begin{array}{c|rrrrrr}
 k\bmod6&0&1&2&3&4&5\\\hline
 S_k(128)\bmod7&5&3&5&2&3&5.
\end{array}
\]
None of these residues belongs to $\mathcal R_3(7)=\{0,1,6\}$.
This example also shows why complete blocks must be included when
the power-sum residue is evaluated.

The case $x=9800$ illustrates the role of prime powers among the
moduli. Since $S_1(9800)=6930^2$, a sieve using only prime moduli
always retains a class congruent to $1$ modulo every corresponding
unit-group period. The modulus $625$ can eliminate such classes.
Indeed, for $k\ge6$ with $k\equiv1\pmod{500}$, the multiples of $5$
contribute zero modulo $625$, while the other bases contribute
their first powers. Hence
\[
 S_k(9800)\equiv\frac{9800\cdot9801}{2}
       -5\frac{1960\cdot1961}{2}
 =38\,416\,000\equiv375\pmod{625}.
\]
The residue has $5$-adic valuation $3$, so it is not a square
modulo $625$. The full list $625,61,31$, together with the
initial classes modulo $8$, excludes all $k\ge6$ at $x=9800$.

\subsection{The exponents \texorpdfstring{$1\le k\le5$}{1 <= k <= 5}}
For each of these five exponents, a separate Python program computes $S_k(x)$
by exact integer arithmetic for all $1\le x\le10^6$.
At each admissible $x$, it tests every divisor $n\ge2$ of
\[
 \begin{cases}
 b,&k=1\text{ or }k\text{ is even},\\
 2b,&k\ge3\text{ is odd}.
 \end{cases}
\]
For each such $n$, it computes $\lfloor S_k(x)^{1/n}\rfloor$ by
an integer-root algorithm and tests equality after raising to the
$n$-th power. Corollary~\ref{cor:exponents} shows that these tests
cover every possible $n$, including composite exponents.
There are $2\,702\,326$ integer-root tests in total. Each returned
integer root $u$ is also checked against
$u^n\le S_k(x)<(u+1)^n$. The accumulated power sums are also
compared with their closed polynomial formulas at every admissible
$x$, giving $2\,500\,000$ additional exact equality checks.
Table~\ref{tab:small-k} records the
number of tests and solutions separately for each $k$.

\begin{table}[htbp]
\centering
\caption{Exact small-exponent checks for $3\le x\le10^6$,
$x\equiv0,3\pmod4$.}\label{tab:small-k}
\begin{tabular}{crrr}
\toprule
$k$&Integer-root tests&Solution quadruples&Wall time (s)\\
\midrule
1&303346&4&0.750\\
2&303346&1&0.772\\
3&896144&500004&2.184\\
4&303346&0&0.944\\
5&896144&0&2.076\\
\midrule
Total&2702326&500009&6.727\\
\bottomrule
\end{tabular}
\end{table}

The resulting set consists of exactly the $500\,000$ quadruples
in~\eqref{eq:universal} and the nine quadruples
in~\eqref{eq:exceptions}. Their existence follows directly
from~\eqref{eq:cubes} and
\begin{gather*}
 S_1(8)=36=6^2,\qquad S_1(288)=41616=204^2,\\
 S_1(9800)=48\,024\,900=6930^2,\qquad S_2(24)=4900=70^2,\\
 S_1(332928)=55\,420\,693\,056=235416^2.
\end{gather*}
Together with the preceding subsection, this proves
Theorem~\ref{thm:finite}.\hfill$\square$

\subsection{A worked example: \texorpdfstring{$x=224$}{x = 224}}
\label{subsec:example224}
We illustrate the complete method by solving~\eqref{eq:main} for
$x=224$. After the initial $2$-adic restriction, only the two
moduli $13$ and $81$ are needed. We show directly that the only
solution is
\[
 (k,n,y)=(3,2,25200).
\]

\paragraph{Reduction to squares.}
Since $224=2^5\cdot7$ and $225=3^2\cdot5^2$, we have
$a=v_2(224\cdot225)=5$ and $b=a-1=4$.
Lemma~\ref{lem:two} gives
\[
 v_2(S_k(224))=
 \begin{cases}
 4,&k=1\text{ or }k\text{ is even},\\
 8,&k\ge3\text{ is odd}.
 \end{cases}
\]
Thus a solution requires $n\in\{2,4\}$ in the first case and
$n\in\{2,4,8\}$ in the second. Every such $n$ is even, so every
perfect-power value under consideration must be a square.
Equivalently, the only prime exponent required by
Corollary~\ref{cor:exponents} is $r=2$.

\paragraph{The initial exponent classes.}
In Lemma~\ref{lem:uniform-two-square} we have $a=5$ and $d=7$.
For even $k\ge4$,
\[
 \frac{S_k(224)}{16}\equiv
 \begin{cases}
 7\pmod8,&k\equiv0\pmod4,\\
 3\pmod8,&k\equiv2\pmod4.
 \end{cases}
\]
Neither residue can be an odd square. For odd $k\ge5$, the same
lemma gives
\[
 \begin{array}{c|rrrr}
 k\bmod8&1&3&5&7\\\hline
 S_k(224)/256\bmod8&1&7&5&3.
 \end{array}
\]
Hence every square value with $k\ge6$ requires $k\equiv1\pmod8$.
With $K=6$, the initial data in
Proposition~\ref{prop:certificate} are therefore
$M_0=8$ and $A_0=\{1\}$. The displayed congruences also exclude
$k=4$ and $k=5$.

\paragraph{The modulus $13$.}
For odd $k$, pairing $j$ and $13-j$ shows that a complete block
of residues has power sum zero modulo $13$.
Since $224=17\cdot13+3$, it follows that
\[
 S_k(224)\equiv1+2^k+3^k\pmod{13}.
\]
The period in $k$ is $12$. Combining it with the initial period
$8$ gives $M_1=24$ and the following three possibilities:
\[
 \begin{array}{c|rrr}
 k\bmod24&1&9&17\\\hline
 S_k(224)\bmod13&6&7&3.
 \end{array}
\]
The square residues modulo $13$ are
$\{0,1,3,4,9,10,12\}$. Consequently only
$k\equiv17\pmod{24}$ survives, and $A_1=\{17\}$.
In particular, every remaining exponent satisfies $k\equiv2\pmod3$.

\paragraph{The modulus $81$.}
For every odd $k\ge3$, we have the uniform congruence
\[
 S_k(224)\equiv27k\pmod{81}.
\]
To prove it, pair $j$ with $225-j$. Since $225^2$ is divisible
by $81$, the binomial theorem gives
\[
 2S_k(224)
 =\sum_{j=1}^{224}\bigl(j^k+(225-j)^k\bigr)
 \equiv225k\,S_{k-1}(224)\pmod{81}.
\]
For every positive even integer $s=2h$,
\[
 \sum_{u=0}^{8}u^s
 \equiv2(1+4^h+7^h)
 \equiv2\bigl(1+(1+3h)+(1+6h)\bigr)
 \equiv6\pmod9.
\]
The terms divisible by $3$ vanish here, and the second congruence
follows by expanding $4^h=(1+3)^h$ and $7^h=(1+6)^h$ modulo $9$.
Splitting $1,\ldots,225$ into $25$ complete blocks modulo $9$,
and noting that $225^{k-1}\equiv0\pmod9$, yields
\[
 S_{k-1}(224)\equiv25\cdot6\equiv6\pmod9.
\]
It follows that
$2S_k(224)\equiv225k\cdot6\equiv54k\pmod{81}$.
Division by $2$, which is invertible modulo $81$, proves the claim.

For the remaining class $k\equiv17\pmod{24}$, this congruence gives
$S_k(224)\equiv54\pmod{81}$.
Every integer $54+81t=27(2+3t)$ has $3$-adic valuation $3$,
so this residue cannot be a square. Thus no exponent $k\ge6$
survives.

\paragraph{The small exponents.}
We have already excluded $k=4,5$. Direct evaluation gives
\[
 S_1(224)=25200=16\cdot1575,\qquad
 S_2(224)=3771600=16\cdot235725.
\]
Their normalized odd parts are $7$ and $5$ modulo $8$,
respectively, so neither value is a square.
Finally,
\[
 S_3(224)=25200^2,\qquad
 25200=2^4\cdot3^2\cdot5^2\cdot7.
\]
Since $v_7(S_3(224))=2$, a representation $S_3(224)=y^n$
with $n\ge2$ requires $n=2$. Hence the only solution is
$(k,n,y)=(3,2,25200)$, as asserted.

\paragraph{The finite certificate.}
For $k\ge6$, the successive periods and surviving classes are
\begin{center}
\begin{tabular}{lcc}
\toprule
Condition&Combined period&Surviving classes\\
\midrule
Initial square restriction&$8$&$\{1\}$\\
After the modulus $13$&$24$&$\{17\}$\\
After the modulus $81$&$216$&$\varnothing$\\
\bottomrule
\end{tabular}
\end{center}
Here $\lambda(81)=54$ and $\lcm(24,54)=216$.
At the last step, the nine lifts $17+24t$ with $0\le t\le8$
all give residue $54$ modulo $81$. The empty final set therefore
excludes every $k\ge6$, with no upper bound on $k$.

\subsection{Reproducibility and computational results}
\label{subsec:computations}
The electronic supplement contains the complete certificate file,
the independent verifier \path{verify_million.cpp}, the exact
small-exponent check \path{small_exponents.py}, and the search
programs. It also includes the run logs, machine-readable reports,
and a script that records the execution environment.

\paragraph{Certificate construction.}
The complete certificate file was constructed from scratch, without
seed certificates. The search processed batches of $1000$ branches
using six worker processes. Its initial pool consisted of $58$
prime-power moduli, all at most $729$.
It produced $591\,767$ certificates and left $37$ branches unresolved.
For these branches, a second search used all odd prime powers
$q\le1000$ whose prime-power exponent is at most $6$.
All $37$ branches then closed. Both stages are implemented by
\path{generate_certificates.py}. The search used a bound of
$20\,000\,000$ on the combined uncompressed period and a bound of
$5000$ on the number of classes in each proposed intersection,
using the compressed input periods and before further period compression.
These are search restrictions; they are not assumptions of the verification.
In particular, the final verifier stores uncompressed survivor sets,
whose maximum size is $9315$.

The search precomputes modular prefix sums and may replace a periodic
set by an exactly equivalent set with a smaller period. A floating-point
score is used only to rank candidate moduli. Neither the score nor the
period-compression procedure is used by the independent verifier.

\paragraph{Independent verification.}
The verifier reconstructs the full set of required $(x,r)$ pairs
from the valuation formula and compares it, in order, with the
certificate records. Missing, additional or misordered records cause
failure. It checks that every modulus is a prime power with exponent
at most $6$, independently computes the modular sums and power-residue
tests, and enumerates every lift in~\eqref{eq:recurrence}.
It accepts a branch only when the final set is empty.

To avoid repeating modular exponentiations, the verifier caches prefix
sums for a modulus and an exponent class. If $x=u q+t$ with $0\le t<q$,
it evaluates
\[
 S_k(x)\equiv u\sum_{j=1}^{q-1}j^k+\sum_{j=1}^{t}j^k\pmod q.
\]
Thus complete blocks are retained, including for even $k$.
This arithmetic cache does not identify or omit any exponent classes.
The proof-critical verification uses exact integer arithmetic, with
explicit overflow checks. The small-exponent program uses Python's
arbitrary-precision integers and checks all permissible composite
exponents as well as prime exponents.

\begin{table}[!ht]
\centering
\caption{Results of the complete independent verification.}
\label{tab:verification}
\begin{tabular}{lr}
\toprule
Quantity&Value\\
\midrule
Admissible values of $x$&500000\\
Verified prime-exponent branches&591804\\
Unresolved branches&0\\
Modulus entries&1221859\\
Distinct moduli&87\\
Largest modulus&991\\
Largest combined period&19656000\\
Largest uncompressed survivor set&9315\\
Residue-class lifts examined&74389941\\
Modular power-residue tests&39009676\\
Small-exponent integer-root tests&2702326\\
Closed-form power-sum cross-checks&2500000\\
Total solution quadruples&500009\\
\bottomrule
\end{tabular}
\end{table}

\paragraph{Hardware and software.}
The computations were performed on 24 October 2026 in a containerized
x86\_64 Linux environment reporting an AMD EPYC 9V74 processor.
The environment exposed nine logical CPUs, with a cgroup CPU quota
of eight CPU-equivalents and a memory limit of $8$~GiB. These are
container resource limits, rather than a claim of dedicated physical
cores. The operating system was Ubuntu 24.04.3 LTS, with Linux kernel
6.18.55.

Certificate construction used CPython 3.12.14 and NumPy 2.3.5.
The initial search used six forked worker processes; the OpenMP and
OpenBLAS thread limits were set to one. The completion search ran in
a single process. The independent verifier was compiled with GCC
13.3.0, using \texttt{-O3 -std=c++17 -Wall -Wextra -Wpedantic}
and \texttt{-Wconversion -Wshadow}, and ran in a single thread.
The small-exponent check used one Python
process and only the standard library.

\begin{table}[!ht]
\centering
\caption{Measured wall-clock times in the recorded environment.
Compilation and manuscript preparation are excluded.}
\label{tab:timings}
\begin{tabular}{lrr}
\toprule
Stage&Worker processes&Wall time (s)\\
\midrule
Full certificate construction&6, then 1&175.889\\
Independent verification, $k\ge6$&1&13.231\\
Exact checks, $1\le k\le5$&1&6.727\\
\midrule
Sum of the three stages&&195.847\\
\bottomrule
\end{tabular}
\end{table}

Within the construction stage, the initial search took $173.305$
seconds and the completion search took $1.401$ seconds. The full
construction time also includes precomputation, branch enumeration,
and assembly of the final file.
The summed CPU time of the timed initial-search worker intervals was
$1035.730$ seconds, and the largest individual worker peak resident set
was $129\,844$~KiB. The latter is not an aggregate memory measurement,
since forked workers share the precomputed tables. The verifier used
$13.191$ CPU seconds and had a peak resident set of $16\,768$~KiB;
the small-exponent check used $6.683$ CPU seconds and peaked at
$9600$~KiB. All times are measurements of these complete runs, not
extrapolations from samples. Detailed reports retain the measurements
at greater precision.

\paragraph{Files and commands.}
Each line of \path{certificates.tsv} has the form
\texttt{x r t q1 ... qt}, specifying the branch and its ordered list
of $t$ moduli. The final file contains $591\,804$ lines and
$10\,222\,900$ bytes. Its SHA-256 digest is
\begin{verbatim}
ba5a41511a90321b1296a407fef30ab77efe710e2d0a5ebb8ac31d2586f06bc1
\end{verbatim}
The two proof checks can be reproduced without rerunning the search:
\begin{verbatim}
g++ -O3 -std=c++17 verify_million.cpp -o verify_million
./verify_million --limit 1000000 certificates.tsv
python3 small_exponents.py
\end{verbatim}
The first command sequence checks every $k\ge6$; the Python program
checks all remaining cases and compares the result with the exact
solution set in~\eqref{eq:universal}--\eqref{eq:exceptions}.
\Needspace{9\baselineskip}
The construction can also be repeated using the supplied search
sources and the pinned dependency in \path{requirements-search.txt}:
\begin{verbatim}
OPENBLAS_NUM_THREADS=1 OMP_NUM_THREADS=1 \
python3 generate_certificates.py --limit 1000000 --workers 6 \
    --output regeneration/certificates.tsv
\end{verbatim}
The regenerated file has the same SHA-256 digest displayed above.
The script \path{verify_progression.py} separately checks the finite
tables used in the next section. The supplement's README gives
the complete commands and explains the machine-readable reports.

\begin{remark}[Availability of the programs]
The programs used for the computations, together with the complete
certificates and detailed run logs, are available from the author
upon request.
\end{remark}

\section{An infinite progression of completely solved cases}
\label{sec:progression}
\begin{proof}[Proof of Theorem~\ref{thm:progression}]
We have $x\equiv4\pmod8$, so $v_2(x(x+1))=2$.
Lemma~\ref{lem:two} therefore excludes $k=1$ and even $k$.
For odd $k\ge3$, it gives $v_2(S_k(x))=2$, and hence $n=2$.

Suppose first that $k\ge5$. The following table holds for every
$x\equiv4\pmod8$:
\begin{equation}\label{eq:table32}
\begin{array}{c|rrrr}
 k\bmod8&1&3&5&7\\\hline
 S_k(x)\bmod32&4&28&20&12.
\end{array}
\end{equation}
To see the uniformity in $x$, increase $x$ by $8$. The four even
bases contribute zero modulo $32$. The four odd bases can be paired
so that both pairs have the same sum $c=2x+8$, which is divisible
by $16$. For odd $u$ and odd $k$,
\[
 u^k+(c-u)^k\equiv kcu^{k-1}\equiv c\pmod{32}.
\]
The two pairs therefore contribute $2c\equiv0\pmod{32}$.
It remains to calculate the table at $x=4$, where
$S_k(4)\equiv1+3^k\pmod{32}$.
Only $4$ among the entries in~\eqref{eq:table32} is a square
modulo $32$. Thus
\begin{equation}\label{eq:k8}
 k\equiv1\pmod8.
\end{equation}

For an odd prime $q$ and odd $k$, a full block of $q$ residues has
power sum zero modulo $q$, by pairing $a$ and $q-a$.
Since $x\equiv28\pmod q$ for $q=13,157,313$, we may replace
$S_k(x)$ by $S_k(28)$ modulo each of these primes.
Combining~\eqref{eq:k8} with the period $12$ modulo $13$ gives
\[
\begin{array}{c|rrr}
 k\bmod24&1&9&17\\\hline
 S_k(28)\bmod13&3&6&7.
\end{array}
\]
Only $3$ is a quadratic residue modulo $13$, so
$k\equiv1\pmod{24}$.

The period modulo $157$ is $156$, and $\lcm(24,156)=312$.
Table~\ref{tab:157} lists all thirteen remaining exponent classes.
Only the classes
\[
 k\equiv193\quad\text{or}\quad217\pmod{312}
\]
survive. The period modulo $313$ is also $312$, and these two
classes give
\[
 S_{193}(28)\equiv89,\qquad S_{217}(28)\equiv178\pmod{313}.
\]
Both residues are quadratic nonresidues, since
\[
 89^{156}\equiv178^{156}\equiv-1\pmod{313}.
\]
This excludes every $k\ge5$. For $k=3$, identity~\eqref{eq:cubes}
gives precisely~\eqref{eq:universal}.
\end{proof}

\begin{table}[htbp]
\centering
\caption{The thirteen classes $k\equiv1\pmod{24}$ modulo $312$.
The last column is the Legendre symbol modulo $157$.}
\label{tab:157}
\begin{tabular}{rrc}
\toprule
$k\bmod312$&$S_k(28)\bmod157$&Legendre symbol\\\midrule
1&92&$-1$\\25&80&$-1$\\49&32&$-1$\\73&103&$-1$\\
97&6&$-1$\\121&54&$-1$\\145&134&$-1$\\169&104&$-1$\\
193&33&$+1$\\217&146&$+1$\\241&65&$-1$\\265&77&$-1$\\289&62&$-1$\\
\bottomrule
\end{tabular}
\end{table}

\section{Odd-prime valuations and further local obstructions}
\label{sec:odd}
Exact valuations of power sums have been studied in several forms;
see Lengyel \cite{L}, MacMillan--Sondow \cite{MS},
Sondow--Tsukerman \cite{ST}, Hajdu \cite{H15,H16}, and
Amdeberhan--Moll--Sharma--Vignat \cite{AMSV}. The formula below
for $p-1\mid k-1$ is a consequence of Lengyel's
Theorem~3. We include a pairing proof, since this form is convenient
for the local restrictions on~\eqref{eq:main}. For $p=3$, it
recovers the odd-exponent part of Hajdu's $3$-adic formula.

\subsection{An exact valuation formula}
For an integer $m\ge2$, write
\[
 A_j(m)=\sum_{i=1}^{m-1}i^j\quad(j\ge1),\qquad A_0(m)=m-1.
\]
The following complete-block congruence is also given in
\cite[Corollary~2]{ST}; we recall an elementary proof.
\begin{lemma}\label{lem:blocks}
Let $p$ be an odd prime, $a=v_p(m)\ge1$, and $j\ge1$. Then
\begin{equation}\label{eq:blocks}
 A_j(m)\equiv
 \begin{cases}
 -(m/p^a)p^{a-1}\pmod{p^a},&p-1\mid j,\\
 0\pmod{p^a},&p-1\nmid j.
 \end{cases}
\end{equation}
In particular, $v_p(A_j(m))\ge a-1$, with equality in the first case.
\end{lemma}
\begin{proof}
For one block modulo $p$, the sum of the $j$-th powers is $-1$ if
$p-1\mid j$, and $0$ otherwise. For $a\ge2$, write each residue
modulo $p^a$ uniquely as $u+vp^{a-1}$, with
$0\le u<p^{a-1}$ and $0\le v<p$. In the binomial expansion,
the constant terms give $p$ times the sum at level $a-1$.
The sum of the linear terms contains the factor
\[
 p^{a-1}\sum_{v=0}^{p-1}v=p^a(p-1)/2,
\]
and every higher term contains $p^{2(a-1)}$, divisible by $p^a$.
Induction proves the result for one block of length $p^a$.
There are $m/p^a$ complete blocks, identical modulo $p^a$.
The residue $0$ contributes nothing, so~\eqref{eq:blocks} follows.
\end{proof}

\begin{proposition}[A consequence of Lengyel's theorem]\label{thm:odd-prime}
Let $p$ be an odd prime and $k\ge3$. If
\[
 p-1\mid k-1,\qquad p\mid x(x+1),
\]
then
\begin{equation}\label{eq:odd-prime}
 v_p(S_k(x))=v_p(k)+2v_p(x(x+1))-1.
\end{equation}
\end{proposition}
\begin{proof}
The divisibility condition makes $k$ odd. First suppose $p\mid m$,
and put $a=v_p(m)$ and $t=v_p(k)$. Since the map $i\mapsto m-i$
permutes $1,\ldots,m-1$, binomial expansion gives
\begin{equation}\label{eq:pairing}
 2A_k(m)=kmA_{k-1}(m)
 +\sum_{j=2}^{k}(-1)^{k-j}\binom{k}{j}m^jA_{k-j}(m).
\end{equation}
By Lemma~\ref{lem:blocks}, the first term has valuation $t+2a-1$.
For $2\le j\le k-1$, the identity
$\binom{k}{j}=(k/j)\binom{k-1}{j-1}$ gives
\[
 v_p\!\left(\binom{k}{j}\right)\ge t-v_p(j).
\]
The corresponding summand in~\eqref{eq:pairing} therefore has
valuation at least
\[
 t+(j+1)a-1-v_p(j)\ge t+2a.
\]
Here we used $(j-1)a-v_p(j)\ge1$, valid for every odd prime $p$,
$j\ge2$, and $a\ge1$. The last summand, with $j=k$, has
valuation $ka\ge t+2a$, since $t\le k-2$.
Thus the first term has strictly smallest valuation.
As $p$ is odd,
\[
 v_p(A_k(m))=t+2a-1.
\]
The term $m^k$ has larger valuation, so the same formula holds for
$A_k(m)+m^k$.
If $p\mid x$, take $m=x$ and use $S_k(x)=A_k(x)+x^k$.
If $p\mid x+1$, take $m=x+1$ and use $S_k(x)=A_k(x+1)$.
Since $x$ and $x+1$ are coprime, these two cases prove the proposition.
\end{proof}

To see the relation with \cite[Theorem~3]{L}, the von
Staudt--Clausen theorem gives $v_p(B_{k-1})=-1$ when
$p-1\mid k-1$. Taking $m=1$ in Lengyel's theorem yields
$v_p(S_k(b-1))=2v_p(b)+v_p(k)-1$ for $p\mid b$.
This gives the case $p\mid x+1$ directly; adding $x^k$ gives the
case $p\mid x$, as in the proof above.

For $p=3$, the condition $p-1\mid k-1$ holds for every odd $k$.
The remaining residue class $x\equiv1\pmod3$ satisfies
$S_k(x)\equiv1\pmod3$, since $i^k\equiv i\pmod3$ for odd $k$.
Thus Proposition~\ref{thm:odd-prime} gives a short independent proof of
the complete odd-exponent part of Hajdu's $3$-adic formula.

\begin{corollary}\label{cor:odd-prime}
If~\eqref{eq:main} holds and $p$ satisfies the hypotheses of
Proposition~\ref{thm:odd-prime}, then
\[
 n\mid v_p(k)+2v_p(x(x+1))-1.
\]
Consequently, for $k\ge3$ with $p-1\mid k-1$:
\begin{enumerate}
\item if $n$ is even and $v_p(k)$ is even, then $p\nmid x(x+1)$;
\item if $p\nmid k$ and $v_p(x(x+1))=1$, there is no solution.
\end{enumerate}
For several applicable primes dividing $x(x+1)$, the exponent $n$
divides the greatest common divisor of the resulting valuations.
\end{corollary}
\begin{proof}
The valuation of $y^n$ is divisible by $n$. In the first excluded
case the valuation in~\eqref{eq:odd-prime} is odd, and in the
second case it equals $1$.
\end{proof}

\subsection{Uniform restrictions for progressions of exponents}
\begin{theorem}\label{thm:density}
Let $\mathcal P$ be a nonempty finite set of odd primes, and put
\[
 M=\lcm_{p\in\mathcal P}\bigl(p(p-1)\bigr).
\]
For every $k>1$ with $k\equiv1\pmod M$, the following hold.
\begin{enumerate}
\item If $S_k(x)$ is a perfect power, then
$v_p(x(x+1))\ne1$ for every $p\in\mathcal P$.
The periodic set defined by these necessary conditions has natural
density
\[
 D(\mathcal P)=\prod_{p\in\mathcal P}
 \left(1-\frac2p+\frac2{p^2}\right).
\]
\item If $S_k(x)$ is an even power, then
$\gcd(x(x+1),\prod_{p\in\mathcal P}p)=1$.
The periodic set defined by this condition has natural density
\[
 E(\mathcal P)=\prod_{p\in\mathcal P}\left(1-\frac2p\right).
\]
\end{enumerate}
\end{theorem}
\begin{proof}
For every $p\in\mathcal P$, the condition on $k$ implies
$p-1\mid k-1$ and $p\nmid k$. Apply
Corollary~\ref{cor:odd-prime}.
Modulo $p^2$, exactly $2(p-1)$ residue classes satisfy
$v_p(x(x+1))=1$. Modulo $p$, exactly two classes have
$p\mid x(x+1)$. The density formulas follow from the Chinese
remainder theorem.
\end{proof}

\begin{corollary}\label{cor:small-density}
For every $\varepsilon>0$, there is an arithmetic progression of
exponents $k>1$ and a periodic set $\mathcal A\subseteq\N$ of
natural density less than $\varepsilon$ such that, simultaneously
for all $k$ in that progression,
\[
 S_k(x)\text{ is a perfect power}\quad\Longrightarrow\quad x\in\mathcal A.
\]
\end{corollary}
\begin{proof}
As $\mathcal P$ increases through the odd primes, $D(\mathcal P)$
tends to zero. Indeed, for $p\ge3$,
$1-2/p+2/p^2\le\exp(-1/p)$, and the sum of the reciprocals of the
primes diverges. Choose a finite $\mathcal P$ with
$D(\mathcal P)<\varepsilon$ and use Theorem~\ref{thm:density}.
\end{proof}

The progression in Corollary~\ref{cor:small-density} depends on
$\varepsilon$. The same periodic set works for every exponent in
the chosen progression. The densities above describe locally
admissible sets; no equality with the density of actual solutions
is asserted. The parameters differ from those in the density theorem
of Patel and Siksek \cite{PS}, which fixes an even summand exponent
and a nonzero common difference and proves that the lengths admitting
a perfect-power value have natural density zero. Gy\H{o}ry, Patel,
Pint\'er and Siksek \cite{GPPS26} establish a character-twisted
analogue: for a primitive quadratic character of conductor $f$,
assuming that the corresponding generalized Bernoulli polynomial
is irreducible of degree at least two, their nonexistence result
holds for a set of lengths of relative density one among the
multiples of $f$. Both results fix the summand exponent. In contrast,
the local restrictions here hold simultaneously for all $k$ in the
chosen arithmetic progression.

\begin{example}
For $k\equiv13\pmod{420}$, the primes $3,5,7$ satisfy
$p-1\mid k-1$ and $p\nmid k$. Any even-power value requires
\[
 \gcd(x(x+1),105)=1,
\]
a condition of density $1/7$. For arbitrary perfect powers,
the corresponding necessary conditions leave density
\[
 \frac59\cdot\frac{17}{25}\cdot\frac{37}{49}
 =\frac{629}{2205}.
\]
\end{example}

\subsection{A fixed-exponent obstruction from Bernoulli numerators}
The use of the lowest-degree term of the power-sum polynomial also
underlies the valuation results in \cite{L}. The following criterion
is useful when a prime divides a Bernoulli numerator. Let $B_j$
denote the Bernoulli numbers with $B_1=-1/2$, and let
$\Z_{(p)}=\{a/b\in\Q:p\nmid b\}$.
\begin{proposition}\label{prop:bernoulli}
Let $k\ge5$ be odd, let $p>k+1$ be prime, and put
$e=v_p(B_{k-1})$. If
\[
 a=v_p(x(x+1))\ge1,\qquad 2a>e,
\]
then
\begin{equation}\label{eq:bernoulli}
 v_p(S_k(x))=2a+e.
\end{equation}
If $e$ is odd, these hypotheses exclude every even exponent $n$
in~\eqref{eq:main}.
\end{proposition}
\begin{proof}
Faulhaber's formula gives
\[
 \sum_{i=1}^{z-1}i^k
 =\frac{B_{k+1}(z)-B_{k+1}}{k+1}
 =\frac{kB_{k-1}}2z^2+R(z).
\]
The coefficients of $z$ and $z^3$ vanish because the odd Bernoulli
numbers of index at least $3$ vanish. After cancellation of the
constant term, the Bernoulli numbers in the remaining coefficients
have index at most $k$. Since $p>k+1$, the von Staudt--Clausen
denominator theorem shows that these coefficients are $p$-integral.
Hence
\[
 R(z)\in z^4\Z_{(p)}[z].
\]
If $v_p(z)=a$, the quadratic term has valuation $2a+e$, since
$p\nmid2k$, whereas $v_p(R(z))\ge4a>2a+e$.
Use $z=x+1$ when $p\mid x+1$. When $p\mid x$, use $z=x$ and
add $x^k$, which is absorbed into the remainder because $k\ge5$.
The assertion about even powers follows from the odd valuation.
\end{proof}

For $k=13$, we have $B_{12}=-691/2730$ and
$13B_{12}/2=-691/420$. Thus
\[
 691\mid x(x+1)\quad\Longrightarrow\quad
 v_{691}(S_{13}(x))=2v_{691}(x(x+1))+1.
\]
Together with Proposition~\ref{thm:odd-prime} at $p=3,5,7$, this gives
\[
 S_{13}(x)\text{ is an even power}\quad\Longrightarrow\quad
 \gcd(x(x+1),3\cdot5\cdot7\cdot691)=1.
\]
The resulting periodic set has density
\[
 \frac17\left(1-\frac2{691}\right)=\frac{689}{4837}.
\]
This illustrates how primes dividing Bernoulli numerators can
supplement the primes covered by~\eqref{eq:odd-prime}.

\subsection{Uniform obstructions from irregular Bernoulli numbers}
\label{subsec:kummer-obstructions}
The preceding proposition treats a fixed exponent and requires
$p>k+1$. We now prove Theorem~\ref{thm:kummer-uniform}, which
uses a congruence on $k$ instead and is uniform on an entire
exponent progression. Lengyel's Theorem~5 \cite{L} already
expresses the valuation through Bernoulli numbers when the endpoint
has sufficiently large $p$-adic valuation. Our purpose here is to
give an explicit depth condition that remains valid as $k$ varies.

We use Kummer's congruences in their full form: for an odd prime $p$,
if $u,v$ are positive even integers, $h\ge1$, and
\[
 u\equiv v\pmod{p^{h-1}(p-1)},\qquad p-1\nmid u,
\]
then
\begin{equation}\label{eq:kummer-full}
 (1-p^{u-1})\frac{B_u}{u}
 \equiv(1-p^{v-1})\frac{B_v}{v}\pmod{p^h}.
\end{equation}
See, for example, \cite[(1.1)]{Rosen2018}. Rational congruences
are understood $p$-adically. Keeping the Euler factors avoids any
additional lower bound on the indices in terms of $h$.

\begin{proof}[Proof of Theorem~\ref{thm:kummer-uniform}]
The congruence on $k$ gives
\[
 k-1\equiv r,\qquad k\equiv r+1,\qquad
 k+1\equiv r+2\pmod p.
\]
Since $4\le r\le p-3$, all three numbers are $p$-adic units.
Also, $k$ is odd and $k\ge r+1\ge5$.
Apply~\eqref{eq:kummer-full} with $u=k-1$, $v=r$, and $h=e+1$:
\[
 (1-p^{k-2})\frac{B_{k-1}}{k-1}
 \equiv(1-p^{r-1})\frac{B_r}{r}\pmod{p^{e+1}}.
\]
The right-hand side has valuation exactly $e$. Since the Euler
factors and the two denominators are units, it follows that
\begin{equation}\label{eq:kummer-B-valuation}
 v_p(B_{k-1})=e.
\end{equation}

For positive integers $z$, the sum $A_k(z)=\sum_{i=1}^{z-1}i^k$
is represented by the polynomial
\[
 A_k(z)=\frac{B_{k+1}(z)-B_{k+1}}{k+1}
       =\sum_{j=1}^{k+1}c_jz^j,
 \qquad
 c_j=\frac{1}{k+1}\binom{k+1}{j}B_{k+1-j}.
\]
The odd Bernoulli numbers of index at least $3$ vanish, so
\[
 c_1=c_3=0,\qquad c_2=\frac{kB_{k-1}}2,
 \qquad v_p(c_2)=e.
\]
The von Staudt--Clausen theorem gives $v_p(B_j)\ge-1$ for
every nonzero Bernoulli number. As $p\nmid k+1$, every nonzero
$c_j$ therefore has valuation at least $-1$. Moreover,
\[
 k-3\equiv r-2\pmod{p-1},\qquad 2\le r-2\le p-5,
\]
so $p-1\nmid k-3$. Hence $B_{k-3}$ and $c_4$ are $p$-integral.

If $v_p(z)=a\ge1$, the term $c_4z^4$ belongs to
$p^{4a}\Z_{(p)}$, and every nonzero term $c_jz^j$ with $j\ge5$
has valuation at least $ja-1\ge5a-1\ge4a$. Thus
\[
 A_k(z)-c_2z^2\in p^{4a}\Z_{(p)},
 \qquad v_p(c_2z^2)=2a+e<4a.
\]
The leading term has strictly smaller valuation than the remainder,
and consequently
\begin{equation}\label{eq:kummer-A-valuation}
 v_p(A_k(z))=2a+e.
\end{equation}
If $p\mid x+1$, take $z=x+1$, so $S_k(x)=A_k(z)$.
If $p\mid x$, take $z=x$ and use $S_k(x)=A_k(x)+x^k$.
The added term has valuation $ka\ge5a>2a+e$ and cannot change
\eqref{eq:kummer-A-valuation}. This proves
\eqref{eq:kummer-uniform-valuation} in both cases.

If $e$ is odd, the valuation $2a+e$ is odd and excludes a square.
If $e$ is odd and $x\equiv3,4\pmod8$, then $v_2(x(x+1))=2$ and
Lemma~\ref{lem:two} gives $v_2(S_k(x))=2$. Any representation
$S_k(x)=y^n$ with $n\ge2$ would therefore require $n=2$,
which has already been excluded.
\end{proof}

\begin{remark}\label{rem:kummer-depth}
When $e$ is odd, put $b=(e+1)/2$. The depth condition in
Theorem~\ref{thm:kummer-uniform} is then equivalent to
$x\equiv0,-1\pmod{p^b}$. Thus its nonexistence assertion is
uniform on explicit progressions of both $x$ and $k$.
For $e=1$, it requires only $p\mid x(x+1)$.
\end{remark}

\begin{proof}[Proof of Corollary~\ref{cor:kummer-families}]
The exact Bernoulli values
\[
 B_{32}=-\frac{7\,709\,321\,041\,217}{510},
 \qquad B_{12}=-\frac{691}{2730}
\]
satisfy $v_{37}(B_{32})=v_{691}(B_{12})=1$.
Apply Theorem~\ref{thm:kummer-uniform} first with
$(p,r,e)=(37,32,1)$. The exponent modulus is $37\cdot36=1332$.
The Chinese remainder theorem identifies the simultaneous conditions
\[
 x\equiv3,4\pmod8,\qquad x\equiv0,-1\pmod{37}
\]
with $x\equiv36,147,148,259\pmod{296}$.
For the second family, use $(p,r,e)=(691,12,1)$ and
$691\cdot690=476\,790$. The conditions modulo $8$ and $691$
give $x\equiv691,2763,2764,4836\pmod{5528}$.
\end{proof}

The first family contains the subprogression
\[
 k\equiv33\pmod{1332},\quad k\equiv1\pmod8
 \quad\Longleftrightarrow\quad k\equiv33\pmod{2664}.
\]
It therefore supplies exclusions among exponents satisfying the
necessary condition from Lemma~\ref{lem:uniform-two-square}.
For instance,
\[
 v_2(S_{33}(36))=2,\qquad v_{37}(S_{33}(36))=3,
\]
and these two valuations alone exclude every perfect-power
representation. The second family extends the preceding fixed
example $k=13$, $p=691$ to an entire progression of exponents.

\begin{remark}
For $x\equiv1,2\pmod4$, formula~\eqref{eq:two} gives
$v_2(S_k(x))=0$. The finite prime-exponent reduction used in
Theorem~\ref{thm:finite} is then unavailable. A complete classification in these two classes requires another
source of exponent restrictions, such as the logarithmic-form
estimates used in \cite{BHMP}, before the same sieve can be made
complete. Theorem~\ref{thm:kummer-uniform} nevertheless gives
uniform square exclusions in these classes for its specified
progressions of $k$ and $x$.
\end{remark}

\Needspace{36\baselineskip}
\appendix
\section{Representative lists of certificate moduli}
\label{app:certificates}
Here $r$ denotes the prime exponent of the candidate perfect power.
Together with the initial classes of Section~\ref{sec:sieve},
each listed set of moduli gives an empty final set in
Proposition~\ref{prop:certificate}, with $K=6$.
The full $591\,804$ lists are supplied in \path{certificates.tsv}.
\small
\begin{longtable}{rrp{102mm}}
\toprule $x$&$r$&Moduli\\\midrule
\endfirsthead
\toprule $x$&$r$&Moduli\\\midrule
\endhead
\bottomrule\endfoot
3&2&81, 13\\
28&2&13, 193, 157, 313\\
64&2&625, 61\\
64&5&625, 101\\
128&3&7\\
224&2&729, 13\\
999&2&729, 193, 17\\
4095&11&729, 67, 199\\
4096&11&23\\
5264&3&243, 73, 181, 625, 61\\
9800&2&625, 61, 31\\
10000&2&37, 7\\
10000&3&7\\
16383&13&53, 157\\
16384&13&157, 131\\
262143&17&729, 103, 239, 137\\
262144&17&103, 239\\
266631&2&41, 11, 131, 313, 9, 601, 37, 109, 251, 379, 113, 193, 29\\
332928&2&729, 73\\
332928&3&243, 37, 109\\
786431&17&103, 137, 307\\
786432&17&137, 103\\
1000000&2&17\\
1000000&5&31, 71, 131, 421, 61\\
\end{longtable}
\normalsize
The lists need not be unique. In particular, the direct proof for
$x=224$ in Section~\ref{subsec:example224} uses the alternative
list $13,81$.

\section*{Use of AI tools}
The author used ChatGPT to polish the mathematical English of the
manuscript and to assist with the development of the computational
programs.

\end{document}